\documentclass[11pt,reqno,a4paper]{amsart}

\usepackage[utf8]{inputenc}
\usepackage{cancel}
\usepackage[margin=0.9in]{geometry}
\usepackage[usenames]{color}
\usepackage{amsmath,amsthm,pdfsync,verbatim,graphicx,epstopdf,enumerate,amssymb}
\usepackage[colorlinks=true]{hyperref}
\usepackage{float}
\usepackage{cancel}
\usepackage[framemethod=tikz]{mdframed}
\usepackage{mathtools}
\usepackage{cite}
\usepackage{enumitem}
\mathtoolsset{showonlyrefs}

\hypersetup{allcolors=blue}
\numberwithin{equation}{section}
\newcommand{\I}{\mathrm{i}}

\newcommand{\lb}{\left(}

\newcommand{\rb}{\right)}

\newcommand{\Beq}{\begin{equation}}
	\newcommand{\Eeq}{\end{equation}}
\newcommand{\beq}{\begin{equation*}}
	\newcommand{\eeq}{\end{equation*}}
\newcommand{\bal}{\begin{align}}
	\newcommand{\eal}{\end{align}}

\newcommand{\B}{\beta}
\newcommand{\bp}{\begin{prob}}
	\newcommand{\ep}{\end{prob}}
\newcommand{\bpr}{\begin{proof}}
	\newcommand{\epr}{\end{proof}}

\newcommand{\bel}[1]{\begin{equation}\label{#1}}
	\newcommand{\ee}{\end{equation}}

\allowdisplaybreaks
\newtheorem{theorem}{Theorem}[section]

\newtheorem{lemma}[theorem]{Lemma}
\newtheorem{proposition}[theorem]{Proposition}

\theoremstyle{definition}

\newtheorem{remark}[theorem]{Remark}

\newcommand{\D}{\mathrm{d}}

\newcommand{\Rb}{\mathbb{R}}

\newcommand{\A}{\alpha}

\newcommand{\ve}{\varepsilon}

\newcommand{\Bb}{\mathbb{B}}
\usepackage[usenames]{color}
\usepackage{amsmath,pdfsync,verbatim,graphicx,epstopdf,enumerate}

\newcommand{\Sb}{\mathbb{S}}

\title[A Unified Range Characterization for SMT]{A Unified Range Characterization for the Spherical mean transform}

\author{Pradipta Chatterjee, Nisha Singhal and Abhilash Tushir}
\address {$^{\ast}$Centre for Applicable Mathematics, Tata Institute of Fundamental Research, Bangalore, India
\newline
E-mail:{\tt\ pradipta22@tifrbng.res.in, nisha2020@tifrbng.res.in, abhilash2296@gmail.com}
\newline
Orcid:{\tt\ 0009-0001-0728-7388, 0009-0006-3005-1986, 0009-0001-2870-0979}}

\begin{document}
\subjclass{33C55, 35R30, 44A12, 44A15, 444A20, 45Q05, 92C55}
\keywords{Spherical Radon transform; spherical harmonics; inversion; range characterization}
\begin{abstract}
In this article, we investigate the range characterization for the spherical mean transform (SMT) of functions supported in the unit ball. In earlier works, in the case of odd dimensions, a set of differential conditions was obtained, whereas in the case of even dimensions, integral conditions were obtained. We prove that these conditions that are different based on the parity of dimension are, in fact, equivalent in odd dimensions. This equivalence shows that the integral conditions yield a unified simple range characterization for the SMT that is valid in both even and odd dimensions.
\end{abstract}
\maketitle
\section{Introduction}
The spherical mean transform (SMT) of a function $f$ 
is defined as
\begin{equation}\label{SRTdef}
    \mathcal{R} f(p,t)=\frac{1}{\omega_{n-1}}\int\limits_{\mathbb{S}^{n-1}} f(p+t\theta) \mathrm{d} S (\theta), \quad (p,t)\in \mathbb{S}^{n-1}\times (0,2),
\end{equation}
where $\omega_{n-1}$ denotes the surface area of the unit sphere $\mathbb{S}^{n-1}$ and $\D S$ is the corresponding surface measure. The SMT is extensively studied due to its importance in both theory and applications. It naturally arises in the analysis of partial differential equations (PDEs) such as the wave and Euler–Poisson–Darboux equations, as well as in integral geometry and approximation theory; see \cite{CH_Book,John-book,Rhee,agranovsky1996approximation,agranovsky1996injectivity,AQ2, AQ3, Agranovsky:1999,And,aramyan2020recovering,Finch_2006} and the references therein. Motivated by applications in imaging, considerable attention has been focused on the SMT that integrates a function supported in the unit ball over spheres with centers restricted to the unit sphere.

In this setting, two problems have gained considerable interest, namely the characterization of the range of $\mathcal{R}$ and its closed-form inversion formula. Range characterizations are of significant importance in theoretical studies as well as in applications. In imaging applications, range descriptions are crucial to reduce numerical errors and noise, and to fill in incomplete measurements. In theory, these conditions are used in constructive proofs, such as producing a function in the range with support restrictions. The SMT shows up in the mathematical analysis of photoacoustic and thermoacoustic tomography, seismic imaging, and many other applications, which require simple and efficient inversion formulas to reconstruct the object under consideration.   

The first complete range characterization for the SMT in two dimensions was  obtained in \cite{ref:AmbKuch-range}, where the range characterization was described in terms of smoothness and support properties together with certain moment and vanishing conditions.
 Subsequently, several unified range characterizations were derived in terms of smoothness and support conditions along with certain orthogonality and vanishing conditions in all dimensions in a series of works; see \cite{Agranovsky-Finch-Kuchment-range,Agranovsky-Kuchment-Quinto,AN,Finch_2006} and the references therein. Most of these works exploited the connection between the SMT, and the wave and Darboux equations. A different characterization using half data was given in \cite{KK}. Recently, a much simpler range description was derived in odd dimensions in terms of a differential operator; see \cite{SRTrange_odd}, and in even dimensions in terms of an integral condition; see \cite{SRTrange_even}, along with the usual support and smoothness conditions.

There are numerous findings concerning the inversion of the SMT under various restrictions on the centres, radii, and accessible data. In the case of full data, several explicit and analytic reconstruction procedures have been established; see \cite{norton1980reconstruction,norton1981ultrasonic,xu2002time,Finch-P-R,Finch-Haltmeir-Rakesh_even-inversion,R,K}. Furthermore, the partial data problem, where only restricted measurements are available, has also attracted considerable attention due to its practical relevance in applications such as thermoacoustic and photoacoustic tomography. Some of the important contributions in this direction can be found in \cite{Ambartsoumian2015,Ambartsoumian2018,Ambartsoumian-Zarrad-Lewis,Souvik:2015,Salman,DoKuny}. In a recent article \cite{PC:2026}, the authors derived a more explicit and streamlined inversion formula in odd dimensions compared to earlier works, requiring only the solution of linear ODEs. Furthermore, the resulting inversion algorithm is computationally efficient, making it both practical and straightforward for numerical reconstruction.

In this work, we revisit the range characterization established in \cite{SRTrange_even}. Building on the approach and results developed in \cite{PC:2026} and \cite{SRTrange_odd}, we extend it from even dimensions to all dimensions. This provides a unified simple range condition that encompasses both even and odd dimensions and highlights the structural similarities between the two settings. The necessity of the condition was proven for arbitrary dimensions in \cite{SRTrange_even}. The main thrust of this work is the derivation of its sufficiency in odd dimensions, which is accomplished by establishing the equivalence of the differential condition in \cite{SRTrange_odd} and the integral condition in \cite{SRTrange_even}.
 The intermediate results obtained in the proof of our range description allow us to derive a more compact formulation of the inversion formula obtained in \cite{PC:2026}, which simplifies its representation and may facilitate further analytical and computational applications.

The rest of the paper is organized as follows. Section \ref{sec:prelim} contains some known results that are crucial to present our main results. Sections \ref{sec:mainresult} and \ref{sec:proofs} provide the statements and the proofs of the main results of this article.
\section{Preliminaries}\label{sec:prelim}
In this section, we recall some definitions and state some of the known range characterizations for the SMT in odd and even dimensions. 
\subsection{Spherical harmonics} The spherical harmonics expansion of $f\in C_{c}^{\infty}(\mathbb{B})$ is given by
\begin{equation*}
        f(x)=\sum_{m=0}^{\infty}\sum_{l=0}^{d_{m,n}} f_{m,l}(|x|)Y_{m,l}\left(\frac{x}{|x|}\right),
\end{equation*}
where
\begin{equation}\label{dqandfqs}
 \quad d_{0}=1,\quad d_{m,n}=\frac{(2m+n-2)(n+m-3)!}{m!(n-2)!},\quad \text{and}\quad f_{m,l}(r)=\int_{\mathbb{S}^{n-1}}f(r\theta)\overline{Y}_{m,l}(\theta)\D \theta.
\end{equation}
Similarly, the spherical harmonics expansion of $g=\mathcal{R}f$ is given by    
\begin{equation*}
     g(\theta,t)=\sum_{m=0}^{\infty}\sum_{l=0}^{d_{m,n}} g_{m,l}(t)Y_{m,l}(\theta),
\end{equation*}
where $g_{m,l}\in C_{c}^{\infty}((0,2))$.
\subsection{Egorychev technique}
In this work, we encounter several combinatorial quantities. Their analysis relies on a method introduced by  Egorychev in \cite{Egorychev}. For $n\geq k\geq 0$, the binomial coefficient $\binom{n}{k}$ can be represented using one of the following contour integrals:
\begin{align}
       \binom{n}{k}&=\frac{1}{2\pi i}\int_{|z|=\epsilon}\frac{(1+z)^{n}}{z^{k+1}}\D z\quad \text{with }\epsilon>0\\
       &=\frac{1}{2\pi i}\int_{|z|=\epsilon}\frac{1}{(1-z)^{k+1}z^{n-k+1}}\D z\quad \text{with }0<\epsilon<1.
\end{align}
A straightforward application of the Cauchy residue theorem also yields that $\binom{n}{k}$ vanishes if $n,k\geq 0$ with $n<k$ or if $k<0\leq n$. However, the case when $n$ is negative requires more careful treatment and  such cases will be carefully explained as and when it arises in the paper.
\subsection{Range characterization - odd dimensions}
The following range characterizations for the SMT in odd dimensions are recalled from \cite{SRTrange_odd}.
\begin{theorem}[Radial functions] Let $\mathbb{B}$ denote the unit ball in $\mathbb{R}^n$ for an odd $n \geq 3$, and $k:=(n-3) / 2$. A function $g \in C_c^{\infty}((0,2))$ is representable as $g=\mathcal{R} f$ for a radial function $f \in C_c^{\infty}(\mathbb{B})$ if and only if $h(t):=t^{n-2} g(t)$ satisfies
\begin{equation}\label{cond:odd}
    \left[\mathcal{L}_k h\right](1-t)=\left[\mathcal{L}_k h\right](1+t), \quad \text { for all } t \in[0,1]
\end{equation}
where $\mathcal{L}_k$ is the linear differential operator of order $k$ :
\begin{equation}\label{lkh:operator}
    \mathcal{L}_k=\sum_{l=0}^k \frac{(k+l)!}{(k-l)!l!2^l}(1-t)^{k-l} D^{k-l}, \quad D=\frac{1}{t} \frac{\mathrm{~d}}{\mathrm{~d} t}
\end{equation}
and $\left[\mathcal{L}_k h\right](\cdot)$ denotes evaluation of the function $\mathcal{L}_k h$ at the given point.
\end{theorem}
\begin{theorem}[General functions]\label{generalrange}
     Let $\Bb$ denote the unit ball in $\Rb^n$ for an odd
$n\geq 3$, and $k := \frac{n-3}{2}$. A function $g\in C_c^{\infty}(\Sb^{n-1}\times (0, 2))$ is representable as $g= \mathcal{R}f$ for $f\in C_c^{\infty}(\Bb)$
if and only if for each $(m, l), m\geq 0, 0 \leq l \leq d_{m,n}, h_{m,l}(t) = t^{n-2}g_{m,l}(t)$ satisfies the following two
conditions:
\begin{itemize}
\item there is a function $\phi_{m,l}\in C_c^{\infty}((0,2))$ such that
\begin{equation}
    h_{m,l}(t) =D^m\phi_{m,l}(t),
\end{equation}
\item the function $\phi_{m,l}(t)$ satisfies
\begin{equation}\label{rangeodd:gencase}
    \left[\mathcal{L}_{m+k} \phi_{m,l}\right](1-t)=\left[\mathcal{L}_{m+k}\phi_{m,l}\right](1+t)\quad \text{for all }t\in[0,1].
\end{equation}
\end{itemize}
\end{theorem}
\subsection{Range characterization - even dimensions}
The following range characterizations for the SMT in even dimensions are recalled from \cite{SRTrange_even}.
\begin{theorem}[Radial functions]\label{SRTrangeeven} Let $\mathbb{B}$ denote the unit ball in $\mathbb{R}^n$ for even $n \geq 2$ and $\alpha=\frac{n-2}{2}$. A function $g \in C_c^{\infty}((0,2))$ is representable as $g=\mathcal{R} f$ for a radial function $f \in C_c^{\infty}(\mathbb{B})$ if and only
if $h(t):=t^{n-2} g(t)$ satisfies for each $0<t<1$,
\begin{multline}\label{cond:even}
    \int_0^{1-t} \frac{u h(u)}{u^{2 \alpha}}\left\{\left[(1+u)^2-t^2\right]\left[(1-u)^2-t^2\right]\right\}^{\alpha-\frac{1}{2}} \mathrm{~d} u=\\\int_{1+t}^2 \frac{u h(u)}{u^{2 \alpha}}\left\{\left[(1+u)^2-t^2\right]\left[(1-u)^2-t^2\right]\right\}^{\alpha-\frac{1}{2}} \mathrm{~d} u.
\end{multline}
\end{theorem}
\begin{theorem}[General functions]\label{SRTrangeeven:general} Let $\mathbb{B}$ denote the unit ball in $\mathbb{R}^{n}$ for an even $n\geq 2$.
A function $g\in C_c^{\infty}(\mathbb{S}^{n-1}\times(0,2))$
 is representable as $g=\mathcal{R}f$ for $f\in C^{\infty}(\mathbb{B})$ if and only if for each
 $(m, l), m\geq 0, 0 \leq l \leq d_{m,n}, h_{m,l}(t) = t^{n-2}g_{m,l}(t)$ 
satisfies the following two conditions:
\begin{itemize}
    \item there is a function $\phi_{m,l}\in C_c^{\infty}((0,2))$ such that
\begin{equation}
    h_{m,l}(t) =D^m\phi_{m,l}(t),
\end{equation}
\item the function $\phi_{m,l}(t)$ satisfies the following for $0<t<1$:
\begin{multline}\label{rangecond:even:general}
     \int_0^{1-t} \frac{u \phi_{m,l}(u)}{u^{2 (m+\alpha)}}\left\{\left[(1+u)^2-t^2\right]\left[(1-u)^2-t^2\right]\right\}^{m+\alpha-\frac{1}{2}} \mathrm{d} u=\\\int_{1+t}^2 \frac{u \phi_{m,l}(u)}{u^{2 (m+\alpha)}}\left\{\left[(1+u)^2-t^2\right]\left[(1-u)^2-t^2\right]\right\}^{m+\alpha-\frac{1}{2}} \mathrm{d} u.
\end{multline}
\end{itemize}
\end{theorem}

\section{Main results}\label{sec:mainresult}
This section is devoted to the presentation of the main results of the article.
The following two results extend the range characterization for SMT in even dimensions stated in Theorems \ref{SRTrangeeven} and \ref{SRTrangeeven:general} to arbitrary dimensions.
\begin{theorem}[Radial functions]\label{mtheorem}
    For odd $n\geq 3$, and $\alpha=\frac{n-2}{2}$, the conditions \eqref{cond:odd} and \eqref{cond:even} are equivalent. Hence,  the range characterization in Theorem \ref{SRTrangeeven} is valid in all dimensions.
\end{theorem}
\begin{theorem}[General functions]\label{mtheorem2}
    For odd $n\geq 3$, and $\alpha=\frac{n-2}{2}$, the conditions \eqref{rangeodd:gencase} and \eqref{rangecond:even:general} are equivalent. Hence, the range characterization in Theorem \ref{SRTrangeeven:general} is valid in all dimensions.
\end{theorem}
Further, we derive a more compact formulation of the inversion formulas derived in \cite[Theorem 2.1, Theorem 2.3]{PC:2026} as a consequence of the intermediate results in the proofs of Theorems \ref{mtheorem} and \ref{mtheorem2}.
\begin{theorem}[Explicit inversion for radial functions]\label{mainthm1} Let $\epsilon>0$, $n\geq 3$ be an odd integer, and $k:=\frac{n-3}{2}$. Also, assume that $f\in C_{c}^{\infty}(\mathbb{B})$ and $$h_{k}(t)=\dfrac{4^{k}\omega_{2k+3}t^{2k+1}}{\omega_{2k+2}}\mathcal{R}f(p,t),~~\forall~(p,t)\in \mathbb{S}^{n-1}\times (0,1).$$ Then $f$ can be recovered in the annular region $\mathbb{B}(\epsilon,1)$ by solving
\begin{equation}\label{ODE:op}
			\left[\frac{\D}{\D t}D^{2k}h_{k}\right](t)
			=\frac{(-1)^{k}k!2^{3k}}{t^{2k}} \left[\mathcal{L}_{k}(t^{2k+1}f)\right](1-t),
\end{equation}
with initial conditions $f^{(i)}(1-\epsilon^{\prime})=0$ for $0\leq i\leq k-1$, where $\epsilon^{\prime}>0$ is small enough such that $1-\epsilon^{\prime}$ lies outside the support of $f$.
\end{theorem}
\begin{theorem}[Explicit inversion for general functions]\label{mainthm2}  Let $\epsilon>0$, $n\geq 3$ be an odd integer, and $k:=\frac{n-3}{2}$. Also, assume that $f\in C_{c}^{\infty}(\mathbb{B})$ and for each $(m,l)$ satisfying $m\geq 0$ and $0\leq l\leq d_{m,n}$, we have
$$h_{m,l}(t):=\dfrac{\omega_{2k+3}t^{2k+1}}{\omega_{2k+2}}g_{m, l}(t)\quad \text{and}\quad \widetilde{f}_{m,l}(t)=\frac{f_{m,l}(t)}{t^m}.$$
Then $\widetilde{f}_{m,l}(t)$ can be recovered in the annular region $\mathbb{B}(\epsilon,1)$ by solving  
\begin{equation}
     \left[\frac{d}{dt}D^{m+2k}h_{m,l}\right](t)=\frac{(-1)^{m+k}k!2^{k}}{t^{2(m+k)}}   \left[\mathcal{L}_{m+k}\left(t^{2(m+k)+1}\widetilde{f}_{m, l}\right)\right](1-t),
\end{equation}
with initial conditions $\widetilde{f}_{m,l}^{(i)}(1-\epsilon^{\prime})=0$ for $0\leq i\leq m+ k-1$, where $\epsilon^{\prime}>0$ is small enough such that $1-\epsilon^{\prime}$ lies outside the support of $f$.
\end{theorem}
\section{Proof of main results}\label{sec:proofs}
In this section, we first reformulate the integral condition \eqref{cond:even} into a more convenient form suitable for further analysis. We then establish several preparatory results that will be required in the proof of our main theorems.

We begin by simplifying each condition and then show that they are equivalent at an appropriate stage. Let us rewrite the condition \eqref{cond:even} to make it easier to handle. 

For $n=2k+3$, $\alpha=\frac{n-2}{2}$, and $h(t)=t^{n-2}g(t)=t^{2k+1}g(t)$, it can be rewritten as
\begin{align}
    \int_0^{1-t} ug(u)\left\{\left[(1+u)^2-t^2\right]\left[(1-u)^2-t^2\right]\right\}^{k} \mathrm{d}u=&\int_{1+t}^2 ug(u)\left\{\left[(1+u)^2-t^2\right]\left[(1-u)^2-t^2\right]\right\}^{k} \mathrm{d} u,
    \end{align}
or equivalently,
\begin{equation}\label{sim:cond:even}
    \int_0^{1-t} ug(u)[Q(t,u)]^{k} \mathrm{~d} u=\int_{1+t}^2 ug(u)[Q(t,u)]^{k}\mathrm{~d} u,~~ \text{where } Q(t,u)=((1+t)^2-u^2)(u^2-(1-t)^2).
\end{equation}
\textbf{Notation:} For simplicity, we introduce the following notations, which we will use throughout the article.
\begin{enumerate}
    \item\label{It1} 
   $ H^{L}_{k}(t):=\int_0^{1-t} ug(u)[Q(t,u)]^{k} \mathrm{~d} u$;
  \item\label{It2} 
   $ H^{R}_{k}(t):=\int_{1+t}^{2} ug(u)[Q(t,u)]^{k} \mathrm{~d} u$;
   \item\label{It3} $ G_{k}(t)
    :=\int_{1-t}^{1} ug(u)[Q(t,u)]^{k} \mathrm{~d} u$.
\end{enumerate}
Our aim is to  establish
\begin{equation}\label{mtheorem1Leq}
    H_{k}^{L}(t)=H_{k}^{R}(t)\quad \text{if and only if }\quad \left[\mathcal{L}_k h\right](1-t)=\left[\mathcal{L}_k h\right](1+t).
\end{equation}
The integral $G_{k}(t)$ was extensively analyzed in the derivation of the explicit inversion formula for SMT in odd dimensions in \cite{PC:2026}. Similar methods can also be used to work with the integrals $H^{L}_{k}(t)$ and
$H^{R}_{k}(t)$. 
The following proposition provides a relation between the derivatives of $H^{L}_{k}(t),H^{R}_{k}(t)$ and $G_{k}(t)$.
\begin{proposition}\label{prop1} 
For $k\geq 0$, we have
\begin{equation*}
    \left[\frac{\D}{\D t}D^{2k}H^{L}_{k}\right](t)=-\left[\frac{\D}{\D t}D^{2k}G_{k}\right](t)=\left[\frac{\D}{\D t}D^{2k}H^{R}_{k}\right](-t).
\end{equation*}   
\end{proposition}
    \begin{proof}
       We consider
    \begin{equation}\label{Hkr}
        H^{R}_{k}(-t)=\int_{1-t}^{2} ug(u)[Q(-t,u)]^{k} \mathrm{~d} u=G_{k}(t)+\int_{1}^{2} ug(u)[Q(t,u)]^{k} \mathrm{~d} u.
    \end{equation}
    This gives
    \begin{equation}
       -\left[\frac{\D}{\D t}D^{2k}H^{R}_{k}\right](-t)= \left[\frac{\D}{\D t}D^{2k}G_{k}\right](t).
    \end{equation} 
   Here we used the fact that $Q(t,u)^k$ is a polynomial in $t$ of degree $4k$, and hence the second term in \eqref{Hkr} vanishes upon applying  $\frac{d}{dt}D^{2k}$.
    We now consider 
    \begin{align*}
      \int_{0}^{2}ug(u)[Q(t,u)]^{k}\D u&=        \int_{0}^{1+t}ug(u)[Q(t,u)]^{k}\D u+        \int_{1+t}^{2}ug(u)[Q(t,u)]^{k}\D u=   H^{L}_{k}(-t)+H^{R}_{k}(t).
    \end{align*}
    Upon differentiating and using the chain rule, we get
    \begin{align*}
        0=-\left[\frac{\D}{\D t}D^{2k}H^{L}_{k}\right](-t)+\left[\frac{\D}{\D t}D^{2k}H^{R}_{k}\right](t)\implies  \left[\frac{\D}{\D t}D^{2k}H^{L}_{k}\right](-t)=\left[\frac{\D}{\D t}D^{2k}H^{R}_{k}\right](t).
    \end{align*}
    This completes the proof.
    \end{proof}
We now proceed to compute these derivatives explicitly in the following theorem.
\begin{theorem}\label{mth} For $k\geq 0$ and $h(t)=t^{2k+1}g(t)$, we have
 \begin{equation}\label{mth:eq1}
        \left[\frac{\D}{\D t}D^{2k}H^{L}_{k}\right](t)
			=\frac{(-1)^{k+1}k!4^{k}2^{k}}{t^{2k}}   [\mathcal{L}_{k}h](1-t).
    \end{equation}
\end{theorem}
Once the above theorem is established, the proof of the main Theorem \ref{mtheorem} follows immediately by combining it with Proposition \ref{prop1}, as shown below.
\begin{proof}[Proof of Theorem \ref{mtheorem}] Suppose  $h$ satisfies the integral condition \eqref{cond:even}, that is, $H_{k}^{L}(t)=H_{k}^{R}(t)$. Combining this identity with Proposition \ref{prop1}, we obtain that
\begin{equation}
    \left[\frac{\D}{\D t}D^{2k}H^{L}_{k}\right](t)=\left[\frac{\D}{\D t}D^{2k}H^{R}_{k}\right](t)=\left[\frac{\D}{\D t}D^{2k}H^{L}_{k}\right](-t).
\end{equation}
Combining this relation with Theorem \ref{mth}, we get
\begin{equation}
\frac{(-1)^{k+1}k!4^{k}2^{k}}{t^{2k}}   [\mathcal{L}_{k}h](1-t)=     \left[\frac{\D}{\D t}D^{2k}H^{L}_{k}\right](t)=\left[\frac{\D}{\D t}D^{2k}H^{L}_{k}\right](-t)=\frac{(-1)^{k+1}k!4^{k}2^{k}}{t^{2k}}   [\mathcal{L}_{k}h](1+t),
\end{equation}
whence we obtain
\begin{equation}
    [\mathcal{L}_{k}h](1-t)=      [\mathcal{L}_{k}h](1+t)\,\,\text{for all}\, t\in[0,1].
\end{equation}
Hence, $h$ satisfies the differential condition \eqref{cond:odd}.\\
Conversely, suppose $h$ satisfies the differential condition \eqref{cond:odd}. Note that using Proposition \ref{prop1} together with Theorem \ref{mth}, we have,
\begin{align}
    \left[\frac{\D}{\D t}D^{2k}H^{L}_{k}\right](t)
			=\frac{(-1)^{k+1}k!4^{k}2^{k}}{t^{2k}}   [\mathcal{L}_{k}h](1-t)~~\text{and}~~\left[\frac{\D}{\D t}D^{2k}H^{R}_{k}\right](t)
			=\frac{(-1)^{k+1}k!4^{k}2^{k}}{t^{2k}}   [\mathcal{L}_{k}h](1+t).
\end{align}
Using condition \eqref{cond:odd} for $t\in(0,1]$ we have
\begin{align}
    \left[\frac{\D}{\D t}D^{2k}H^{L}_{k}\right](t)=\left[\frac{\D}{\D t}D^{2k}H^{R}_{k}\right](t)
\end{align}
 Using the support restriction on $g$, we observe that both $H^{L}_{k}(t)$ and $H^{R}_{k}(t)$ vanish for $t$ sufficiently close to $1$. Therefore, the identity $H^{L}_{k}(t)=H^{R}_{k}(t)$ on $t\in(\epsilon,1]$ follows by repeatedly integrating over the interval $[t,1]$. Finally, using its smoothness and letting $\epsilon\to 0$, we conclude that $H^{L}_{k}(t)=H^{R}_{k}(t)$ for all $t\in[0,1]$. This completes the proof.
    \end{proof}
The proof of Theorem \ref{mth} is a straightforward consequence of the following propositions. For the sake of completeness, we will briefly present the argument. In these propositions, we systematically simplify both sides of \eqref{mth:eq1} to a form in which they can be directly compared.
\begin{proposition}\label{prop:mthm1}
    For $k\geq 0$ and $h(t)=t^{2k+1}g(t)$, we have
    \begin{equation}\label{expdDH}
        \left[\frac{\D}{\D t}D^{2k}H^{L}_{k}\right](t)=\frac{(-1)^{k+1}k!4^{k}}{t^{2k}}\left(\frac{(2k)!}{k!}h(1-t)+\sum_{j=1}^{k}P_{k,j}(t)\frac{1}{(1-t)^{2k-j}} h^{(j)}(1-t)\right),
    \end{equation}
    where 
    $P_{k,j}(t)=\sum\limits_{r=0}^{2k-2j}C_{k,j,r}t^{2k-j-r}$ is a polynomial of degree at most $2k-j$, and the coefficient $C_{k,j,r}$ is given by   
        \begin{align*}
         C_{k,j,r} 
   =(-1)^{r}\frac{k!2^{j}}{j!}\sum_{l=0}^{r+j}\sum_{p=0}^{2k+l-j}2^{p}\binom{r+j}{l}\binom{2k-l}{k-l}\binom{2k+l-j}{p}\binom{2l-j-p}{2l-2j-r-p}.
    \end{align*}
\end{proposition}
\begin{proposition}\label{prop:thm2}
For $k\geq 0$ and $h(t)=t^{2k+1}g(t)$, we have
    \begin{equation}\label{explkh}
        [\mathcal{L}_{k}h](1-t)=\frac{1}{2^{k}}\left(\frac{(2k)!}{k!}h(1-t)+\sum_{j=1}^{k}P_{k,j}(t)\frac{1}{(1-t)^{2k-j}} h^{(j)}(1-t)\right),
    \end{equation}
    where $P_{k,j}(t)$ is the polynomial defined in Proposition \ref{prop:mthm1}.
\end{proposition}
    For $k=0$, the contribution of the summations in \eqref{expdDH} and \eqref{explkh} is assumed to be zero.
    
Once these results have been established, the proof of Theorem \ref{mth} follows as an immediate consequence.
\begin{proof}[Proof of Theorem \ref{mth}]
   Noting that the expressions inside the brackets on the right-hand sides of \eqref{expdDH} and \eqref{explkh} are the same, we get
    \begin{equation}
         \left[\frac{\D}{\D t}D^{2k}H^{L}_{k}\right](t)
			=\frac{(-1)^{k+1}k!4^{k}}{t^{2k}}   2^{k}[\mathcal{L}_{k}h](1-t).
    \end{equation}
    This completes the proof of the theorem.
\end{proof}
We now proceed to prove Propositions \ref{prop:mthm1} and \ref{prop:thm2}. It is straightforward to check that \eqref{expdDH} and \eqref{explkh} hold for $k=0$, so in the following proofs, we will focus on $k\geq 1$.
\begin{proof}[Proof of Proposition 
\ref{prop:mthm1}]
This proof will be carried out in the following two steps.
\begin{enumerate}
    \item     We have the integral $H^{L}_{k}(t)$ and thus $\left[\frac{\D}{\D t}D^{2k}H^{L}_{k}\right](t)$ in terms of $g$, but the range condition \eqref{cond:odd} is formulated in terms of $h$. Therefore, as a first step, we use Proposition \ref{prop1} together with \cite[Theorem 2.1]{PC:2026} to rewrite $ \left[\frac{\D}{\D t}D^{2k}H^{L}_{k}\right](t)$ in  terms of  $h$.
    \item In the second step, we simplify the resultant expression in the first step to obtain the desired expression in \eqref{expdDH}.
\end{enumerate}
\subsection*{Step 1: Reformulation of $\left[\frac{\D}{\D t}D^{2k}H^{L}_{k}\right](t)$ in terms of $h$}
Using Proposition \ref{prop1}, and \cite[Theorem 2.1]{PC:2026} in the expression of $\left[\frac{\D}{\D t}D^{2k}G_{k}\right](t)$, we get
\begin{multline}\label{Gkode}
		\left[\frac{\D}{\D t}D^{2k}H^{L}_{k}\right](t)=-\left[\frac{\D}{\D t}D^{2k}G_{k}\right](t)\\	=(-1)^{k+1}k!4^{k}\sum_{m=0}^{k}\left[\sum_{n=m}^{k}\left(\sum_{l=n}^{k}\frac{2^{l+m-n}(2k-l)!}{(k-l)!m!}\binom{l+1}{n+1}\binom{l+n-m}{n-m}\frac{1}{t^{l}}\right)\frac{(1-t)^{n+1}}{t^{n-m}}\right]g^{(m)}(1-t).
		\end{multline}
Applying the Leibniz rule to $g(t)=\dfrac{h(t)}{t^{2k+1}}$, we have 
\begin{align}
    g^{(m)}(1-t)=\sum_{j=0}^{m}(-1)^{m-j}\binom{m}{j}\frac{(m-j+2k)!}{(2k)!}\frac{1}{(1-t)^{m-j+2k+1}}h^{(j)}(1-t).
\end{align}
Substituting this expression of $g^{(m)}(1-t)$ in \eqref{Gkode}, multiplying both sides by $t^{2k}$, and changing the order of summations, we get
\begin{multline*}
     \frac{1}{(-1)^{k+1} k! 4^k}t^{2k} \left[\frac{\D}{\D t}D^{2k}H^{L}_{k}\right](t)\\
   =\sum_{j=0}^{k}\sum_{m=j}^{k}\sum_{n=m}^{k}\sum_{l=n}^{k}\frac{(-1)^{m-j}2^{l+m-n}(2k-l)!(m-j+2k)!}{(k-l)!(2k)!m!}\binom{m}{j}\binom{l+1}{n+1}\binom{l+n-m}{n-m}\frac{t^{2k+m-n-l}}{(1-t)^{m+2k-n-j}}\\\times h^{(j)}(1-t).
\end{multline*}
Replacing the index $l$ by $k-l$ and substituting $n-m=p$, we get
\begin{multline*}
   \frac{1}{(-1)^{k+1} k! 4^k}t^{2k} \left[\frac{\D}{\D t}D^{2k}H^{L}_{k}\right](t) \\ =\sum_{j=0}^{k}\sum_{m=j}^{k}\sum_{p=0}^{k-m}\sum_{l=0}^{k-m-p}\frac{(-1)^{m-j}2^{k-l-p}(k+l)!}{l!j!}\binom{m-j+2k}{2k}\binom{k-l+1}{m+p+1}\binom{k-l+p}{p}\frac{t^{k+l-p}}{(1-t)^{2k-p-j}}\times\\ h^{(j)}(1-t).
\end{multline*}
By changing the order of summations, we obtain
\begin{multline}\label{dD2kHL}
   \frac{1}{(-1)^{k+1} k! 4^k}t^{2k} \left[\frac{\D}{\D t}D^{2k}H^{L}_{k}\right](t) \\ =\sum_{j=0}^{k}\sum_{p=0}^{k-j}\sum_{l=0}^{k-p-j}(-1)^{j}2^{k-l-p}\frac{(k+l)!}{l!j!}\binom{k-l+p}{p}\sum_{m=j}^{k-p-l}(-1)^{m}\binom{m-j+2k}{2k}\binom{k-l+1}{m+p+1}\times\\\frac{t^{k+l-p}}{(1-t)^{-p}}\frac{h^{(j)}(1-t)}{(1-t)^{2k-j}}.
\end{multline}
We now consider the innermost summation and replace the index $m$ with $m-p-1$.
\begin{multline}\label{innsum:m}
     \sum_{m=p+1+j}^{k-l+1}(-1)^{m-p-1}\binom{m-p-1-j+2k}{2k}\binom{k-l+1}{m}=\\(-1)^{p+1}\sum_{m=0}^{k-l+1}(-1)^{m}\binom{m-p-1-j+2k}{2k}\binom{k-l+1}{m}
    =(-1)^{p+k-l}\binom{2k-j-1-p}{k+l-1}.
\end{multline}
Here we have used the fact that $\binom{m-p-1-j+2k}{2k}=0$ for $m<p+1+j$ in order to extend the limit, and the formula 
$$\sum_{r=0}^{n-s}(-1)^r\binom{a+r}{2n-s}\binom{n-s}{r}=(-1)^{n-s}\binom{a}{n},$$
from \cite[Lemma 3.1]{SRTrange_odd} in the last step, with $a=2k-j-1-p, n=k+l-1, s=2l-2$. Substituting the expression \eqref{innsum:m} in \eqref{dD2kHL}, we get
\begin{align*}
&\frac{1}{(-1)^{k+1} k! 4^k}t^{2k} \left[\frac{\D}{\D t}D^{2k}H^{L}_{k}\right](t) \\
     =&\sum_{j=0}^{k}\sum_{p=0}^{k-j}\sum_{l=0}^{k-p-j}\frac{(-1)^{j+p+k-l}2^{k-l-p}k!}{j!}\binom{k+l}{l}\binom{2k-j-1-p}{k+l-1}\binom{k-l+p}{p}t^{k+l-p}(1-t)^{p}\frac{h^j(1-t)}{(1-t)^{2k-j}}.
\end{align*}
Hence, we have
\begin{align}
   \left[\frac{\D}{\D t}D^{2k}H^{L}_{k}\right](t)= \frac{(-1)^{k+1} k! 4^k}{t^{2k}}\left(P_{k,0}(t)h(1-t)+\sum_{j=1}^{k}P_{k,j}(t)\frac{1}{(1-t)^{2k-j}} h^{(j)}(1-t)\right), 
\end{align}
where
\begin{align*}
    P_{k,0}(t)&:=\sum_{p=0}^{k}\sum_{l=0}^{k-p}(-1)^{p+k-l}2^{k-l-p}k!\binom{k+l}{l}\binom{2k-1-p}{k+l-1}\binom{k-l+p}{p}\frac{t^{k+l-p}(1-t)^{p}}{(1-t)^{2k}},\, \text{and}\\
    P_{k,j}(t)&:=\sum_{p=0}^{k-j}\sum_{l=0}^{k-p-j}(-1)^{j+p+k-l}2^{k-l-p}\frac{k!}{j!}\binom{k+l}{l}\binom{2k-j-1-p}{k+l-1}\binom{k-l+p}{p}t^{k+l-p}(1-t)^{p},~~ j\geq 1.
\end{align*}
\subsection*{Step 2: Simplification of $P_{k,0}(t)$ and $P_{k,j}(t)$}
\textbf{Expression of $P_{k,0}(t)$:}
Replacing the index $l$ with $k-l$, we get
\[
P_{k,0}(t)= \sum_{p=0}^{k}\sum_{l=p}^{k}(-1)^{p+l}2^{l-p}k!\binom{2k-l}{k}\binom{2k-1-p}{2k-l-1}\binom{l+p}{p}\frac{t^{2k-l-p}(1-t)^{p}}{(1-t)^{2k}}.
\]
Renaming the indices $p$ as $r$, and $l$ as $p$, we have
\[
P_{k,0}(t)= \sum_{r=0}^{k}\sum_{p=r}^{k}(-1)^{r+p}2^{p-r}k!\binom{2k-p}{k}\binom{2k-1-r}{2k-p-1}\binom{p+r}{r}t^{2k-p-r}(1-t)^{r}\frac{1}{(1-t)^{2k}}.
\]
Let us write $-t$ as $y$. This gives
\begin{align*}
    \frac{P_{k,0}(t)}{k!}&= \frac{y^{2k}}{(1+y)^{2k}} \sum_{r=0}^{k}\sum_{p=r}^{k}2^{p-r}\binom{2k-p}{k}\binom{2k-1-r}{2k-p-1}\binom{p+r}{r}\frac{(1+y)^{r}}{y^{p+r}}.
\end{align*}
Note that this sum is exactly the same as the summation $S_1$ in \cite[page 23]{SRTrange_odd} for $u=s=0$, except for the limits of $p$ and $r$. For the reader’s convenience, we briefly reproduce the relevant steps from that article. 
We can write $P_{k,0}(t)$ in terms of a contour integral as follows:
\begin{align*}
&\frac{P_{k,0}(t)(1+y)^{2k}}{k!y^{2k}}= \sum\limits_{r=0}^{k}\sum\limits_{p=r}^{k}\frac{1}{(2\pi\I)^3}\\
&\times\int\limits_{|z|=\ve_1}\int\limits_{|w|=\ve_2}\int\limits_{|v|=\ve_3}\!\!\! 2^{p-r}\frac{1}{(1-z)^{k+1} z^{k-p+1}}\frac{1}{(1-w)^{p+1}w^{r+1}} \frac{1}{(1-v)^{2k-p} v^{-r+p+1}}\frac{(1+y)^{r}}{y^{p+r}}\D z \D w \D v.
\end{align*}
We adopt the same choice of contours as in \cite{SRTrange_odd}. Note that for $p<r$, the integral in $v$ vanishes by Cauchy's theorem. Similarly for $r>k$, the integral in $v$ is 0, since $p\leq k$. Thus, we can  extend the lower and upper limits of $r$ to $0$ and $\infty$, respectively. We can also extend the upper limit of $p$ to $\infty$ as the integral in $z$ vanishes for $p>k$. Making the change of variable $z(1-z)=\eta$ as done in the aforementioned paper, we obtain
\begin{align*}
    \frac{P_{k,0}(t)}{k!}&=\frac{y^{2k}}{(1+y)^{2k}} \frac{y(1+y)^{2k}}{2\pi\I}\int \frac{1}{\eta^{k+1}}\frac{1}{(\sqrt{1-4\eta}+\sqrt{y^2-4\eta})^{2k}}\frac{1}{\sqrt{y^2-4\eta}\sqrt{1-4\eta}} \D \eta\\
    &=\frac{y^{2k+1}}{2\pi\I} \int \frac{1}{\eta^{k+1}}\frac{1}{(\sqrt{1-4\eta}+\sqrt{y^2-4\eta})^{2k}}\frac{1}{\sqrt{y^2-4\eta}\sqrt{1-4\eta}} \D \eta.
\end{align*}
For ease of notation, we write $\A = \sqrt{1-4\eta},\,\,\B=\sqrt{t^2-4\eta}$.
\[
\frac{P_{k,0}(t)}{k!}= \frac{y^{2k+1}}{2\pi\I} \int \frac{1}{\eta^{k+1}}\frac{1}{(\A+\B)^{2k}}\frac{1}{\A\B} \D \eta.
\]
We next make the change of variable $\A + \B=\delta$ (see \cite[page 25]{SRTrange_odd} for more details). We have
\begin{align*}
   \frac{P_{k,0}(t)}{k!}&=-\frac{2^{4k+3}y^{2k+1}}{2\pi\I}\int \frac{\lb (1-(\delta-y)^2)((\delta+y)^2-1)\rb^{-k-1}}{\delta^{-1}}\D \delta\\
    &=(-1)^{k} \frac{2^{4k+3}y^{2k+1}}{2\pi\I}\int \frac{\delta}{\lb(\delta^2-(y+1)^2)(\delta^2-(y-1)^2)\rb^{k+1}}\D \delta.
\end{align*}
Making another change of variable $\delta^2-(y+1)^2=\B$, we finally have
\begin{align*}
      \frac{P_{k,0}(t)}{k!}&=(-1)^{k} \frac{2^{4k+2}y^{2k+1}}{2\pi\I}  \int \frac{1}{\lb\B(\B+4y)\rb^{k+1}}\D \B=(-1)^{k} \frac{(4y)^{k}}{2\pi\I}  \int \frac{1}{\lb\B\lb\frac{\B}{4y}+1\rb\rb^{k+1}}\D \B\\
    &=(-1)^{k} \frac{(4y)^{k}}{2\pi\I}\int\frac{1}{{\B}^{k+1}}\sum\limits_{p\geq 0} {k+p\choose p} \frac{(-\B)^{p}}{(4y)^{p}} \D \B
    =\binom{2k}{k}.
\end{align*}
Hence, $P_{k,0}(t)=\frac{(2k)!}{k!}$.\\

\medskip
\noindent\textbf{Expression of $P_{k,j}(t)$:}
We have
\begin{align*}
      P_{k,j}(t)&:=\sum_{p=0}^{k-j}\sum_{l=0}^{k-p-j}(-1)^{j+p+k-l}2^{k-l-p}\frac{k!}{j!}\binom{k+l}{l}\binom{2k-j-1-p}{k+l-1}\binom{k-l+p}{p}t^{k+l-p}(1-t)^{p}.
\end{align*}
By expanding $(1-t)^p$, we get
\begin{align*}
      &P_{k,j}(t)\\&=\sum_{p=0}^{k-j}\sum_{l=0}^{k-p-j}\sum_{r=0}^{p}(-1)^{j-r+k-l}2^{k-l-p}\frac{k!}{j!}\binom{k+l}{l}\binom{2k-j-1-p}{k+l-1}\binom{k-l+p}{p}\binom{p}{r}t^{k+l-r}\\
      &=\sum_{p=0}^{k-j}\sum_{l=0}^{k-p-j}\sum_{r=k-j-l}^{k-j-l+p}(-1)^{r}2^{k-l-p}\frac{k!}{j!}\binom{k+l}{l}\binom{2k-j-1-p}{k+l-1}\binom{k-l+p}{p}\binom{p}{r+j-k+l}t^{2k-j-r},
\end{align*}
where the last step follows from the change of index $r\to r+j-k+l$. Noting that $\binom{p}{r+j-k+l}=0$ for $r<k-j-l$ and for $r>k-j-l+p$, we may extend the range of summation in $r$ to $0\leq r\leq 2k-2j$. Then, by changing the order of summation, we obtain
\begin{align*}  
   P_{k,j}(t)
    &=\sum_{r=0}^{2k-2j}\sum_{p=0}^{k-j}\sum_{l=0}^{k-p-j}(-1)^{r}2^{k-l-p}\frac{k!}{j!}\binom{k+l}{l}\binom{2k-j-1-p}{k+l-1}\binom{k-l+p}{p}\binom{p}{r+j-k+l}t^{2k-j-r}\\
    &=\sum\limits_{r=0}^{2k-2j}C_{k,j,r}t^{2k-j-r},
\end{align*}
where the coefficients $C_{k,j,r}$ are given by
\begin{equation}\label{ckjr:firstexp}
   C_{k,j,r}:= (-1)^{r}\frac{k!}{j!}\sum_{l=0}^{k-j}2^{k-l}\binom{k+l}{l}\sum_{p=0}^{k-j-l}2^{-p}\binom{2k-j-p-1}{k+l-1}\binom{k-l+p}{p}\binom{p}{r+j-k+l}.
\end{equation}
Utilizing Lemma \ref{ckjr:lemma}, we obtain the following expression of $C_{k,j,r}$:
\begin{equation}
    C_{k,j,r} 
   =(-1)^{r}\frac{k!2^{j}}{j!}\sum_{l=0}^{r+j}\sum_{p=0}^{2k+l-j}2^{p}\binom{r+j}{l}\binom{2k-l}{k-l}\binom{2k+l-j}{p}\binom{2l-j-p}{2l-2j-r-p}.
\end{equation}
This completes the proof of Proposition \ref{prop:mthm1}.
    
\end{proof}
\begin{proof}[Proof of Proposition \ref{prop:thm2}]
The proof mainly involves the following two steps:
\begin{enumerate}
    \item Since $\mathcal{L}_{k}$ is originally written in terms of $D=\frac{1}{t}\frac{\mathrm{d}}{\mathrm{d}t}$, our first step is to recast it in terms of the usual derivatives using \cite[Lemma 3.3]{PC:2026}. 
    \item In the second step, we simplify the resultant expression in the first step to obtain the desired expression in \eqref{explkh}.
\end{enumerate}
\subsection*{Step 1: Reformulation of the operator $\mathcal{L}_{k}$ in terms of ordinary derivatives}
Following is the relation between $D$-derivatives and ordinary derivatives from \cite[Lemma 3.3]{PC:2026}:
\begin{equation}
    D^{r}=\sum_{j=1}^{r}(-1)^{r-j}\frac{(2r-1-j)!}{2^{r-j}(r-j)!(j-1)!}\frac{1}{t^{2r-j}}\frac{\D^{j}}{\D t^{j}},\quad r\in \mathbb{N}.
\end{equation}
Using this along with \eqref{lkh:operator}, and changing the order of summations, we get
    \begin{align}
        \mathcal{L}_k(h(t))&=\sum_{l=0}^k \frac{(k+l)!}{(k-l)!l!2^l}(1-t)^{k-l} D^{k-l}(h(t))\nonumber\\
        &=\frac{(2k)!}{k!2^{k}}h(t)+\sum_{j=1}^{k}\sum_{l=0}^{k-j}\left((-1)^{k-l-j}\frac{(k+l)!(2k-2l-1-j)!}{2^{k-j}(k-l)!l!(k-l-j)!(j-1)!}\frac{(1-t)^{k-l}}{t^{2k-2l-j}}\right)h^{(j)}(t).
    \end{align}
    This gives
        \begin{align*}
        &2^{k}[\mathcal{L}_kh](1-t)\\
        &=\frac{(2k)!}{k!}h(1-t)+\sum_{j=1}^{k}\sum_{l=0}^{k-j}\frac{(-1)^{k-l-j}2^{j}(k+l)!(2k-2l-1-j)!}{(k-l)!l!(k-l-j)!(j-1)!}\frac{t^{k-l}}{(1-t)^{-2l}}\frac{h^{(j)}(1-t)}{(1-t)^{2k-j}}\\
         &=\frac{(2k)!}{k!}h(1-t)+\sum_{j=1}^{k}\sum_{l=0}^{k-j}\frac{(-1)^{k-l-j}k!2^j}{(2k-2l-j)(j-1)!}\binom{k+l}{l}\binom{2k-2l-j}{k-l}t^{k-l}(1-t)^{2l}\frac{h^{(j)}(1-t)}{(1-t)^{2k-j}}\\
          &=\underbrace{\frac{(2k)!}{k!}}_{Q_{k,0}(t)}h(1-t)+\sum_{j=1}^{k}\underbrace{\sum_{l=0}^{k-j}\sum_{r=0}^{2l}\frac{(-1)^{k-l-j-r}k!2^j}{(2k-2l-j)(j-1)!}\binom{k+l}{l}\binom{2k-2l-j}{k-l}\binom{2l}{r}t^{k+l-r}}_{Q_{k,j}(t)}\frac{h^{(j)}(1-t)}{(1-t)^{2k-j}}\\
          &=Q_{k,0}(t)h(1-t)+\sum_{j=1}^{k}Q_{k,j}(t)\frac{h^{(j)}(1-t)}{(1-t)^{2k-j}}.
    \end{align*}
   Clearly, $ Q_{k,0}(t)=\frac{(2k)!}{k!}=P_{k,0}(t)$. Now our goal is to show that $Q_{k,j}(t)=P_{k,j}(t)$ for $j\geq 1$.
    \subsection*{Step 2: Simplification of coefficients  $Q_{k,j}(t)$}
    Replacing $r$ with $r+j-k+l$, we get
    \begin{align}
      \!\!\!\!Q_{k,j}(t)&=\sum_{l=0}^{k-j}\sum_{r=k-j-l}^{k+l-j}\frac{(-1)^{r}k!2^j}{(2k-2l-j)(j-1)!}\binom{k+l}{l}\binom{2k-2l-j}{k-l}\binom{2l}{r+j-k+l}t^{2k-j-r}\\
      &=\sum_{r=0}^{2k-2j}\underbrace{\sum_{l=0}^{k-j}\frac{(-1)^{r}k!2^j}{(2k-2l-j)(j-1)!}\binom{k+l}{l}\binom{2k-2l-j}{k-l}\binom{2l}{r+j-k+l}}_{B_{k,j,r}}t^{2k-j-r}\\
      &=\sum_{r=0}^{2k-2j}B_{k,j,r}t^{2k-j-r}.\label{Bkj:expres}
    \end{align}
    The goal of proving $Q_{k,j}(t)=P_{k,j}(t)$ now boils down to showing that $B_{k,j,r}=C_{k,j,r}$ which is proved in Lemma \ref{Bkjr:lemma}.
    Thus we get
\begin{align*}
    2^{k}[\mathcal{L}_{k}h](1-t)=\left(\frac{(2k)!}{k!}h(1-t)+\sum_{j=1}^{k}P_{k,j}(t)\frac{1}{(1-t)^{2k-j}}h^{(j)}(1-t)\right),
\end{align*}
where $P_{k,j}(t)=\sum\limits_{r=0}^{2k-2j}C_{k,j,r}t^{2k-j-r}$.
    This completes the proof of Proposition \ref{prop:thm2}.
\end{proof}
Invoking Propositions \ref{prop:mthm1} and \ref{prop:thm2}, the proof of Theorem \ref{mth} follows as outlined earlier in this section. Furthermore, the case of general functions follows immediately, since it is simply the $(m+k)$ analogue of the radial case.
\begin{proof}[Proof of Theorem \ref{mtheorem2}]
    For $n=2k+3$, $\alpha=\frac{n-2}{2}$, and $\phi_{m,l}(t)=t^{2(m+k)+1}\Phi_{m,l}(t)$, the relation \eqref{rangecond:even:general} can be rewritten as

\begin{align*}
    H^{L}_{m+k}(t):=& \int_0^{1-t} u\Phi_{m,l}(u)\left\{\left[(1+u)^2-t^2\right]\left[(1-u)^2-t^2\right]\right\}^{m+k} \mathrm{d} u\\  =&\int_{1+t}^2 u\Phi_{m,l}(u)\left\{\left[(1+u)^2-t^2\right]\left[(1-u)^2-t^2\right]\right\}^{m+k} \mathrm{d} u=:H_{m+k}^{R}(t).
\end{align*}
Accordingly, Theorem \ref{mtheorem2} may be restated in the following form:
\begin{equation}
    H_{m+k}^{L}(t)=H_{m+k}^{R}(t)\quad \text{if and only if }\quad \left[\mathcal{L}_{m+k} \phi_{m,l}\right](1-t)=\left[\mathcal{L}_{m+k} \phi_{m,l}\right](1+t).
\end{equation}
Therefore, proving Theorem \ref{mtheorem2} reduces to showing the $(m+k)$ analogue of the radial case; see \eqref{mtheorem1Leq}.  This completes the proof.
\end{proof}
We now proceed to prove Theorems \ref{mainthm1} and \ref{mainthm2}.
\begin{proof}[Proof of Theorem \ref{mainthm1}]
We adopt the notations used in the proof of \cite[Theorem 2.1]{PC:2026},
\begin{align*}
    G_{k}(t):=h_{k}(t)=\int_{1-t}^{1}uf(u)[Q(t,u)]^{k}\D u\quad \text{and}\quad H_{k}^{L}(t):=\int_{0}^{1-t}uf(u)[Q(t,u)]^{k}\D u,
\end{align*}
and set $F(t)=t^{n-2}f(t)=t^{2k+1}f(t)$.
Using Proposition \ref{prop1} and Theorem \ref{mth}, we have
\begin{align}
    \left[\frac{d}{dt}D^{2k}h_{k}\right](t)=    \left[\frac{d}{dt}D^{2k}G_{k}\right](t)=-\left[\frac{d}{dt}D^{2k}H_{k}\right](t)&=\frac{(-1)^{k}k!4^{k}2^{k}}{t^{2k}}   [\mathcal{L}_{k}F](1-t)\\
    &=\frac{(-1)^{k}k!4^{k}2^{k}}{t^{2k}}   [\mathcal{L}_{k}(t^{2k+1}f)](1-t).
\end{align}
This completes the proof.
\end{proof}
\begin{proof}[Proof of Theorem \ref{mainthm2}] We recall the following expressions from the proof of \cite[Theorem 2.3]{PC:2026}:
\begin{equation}\label{phiqs}
    \phi_{m,l}(t)=\int_{1-t}^1 u \widetilde{f}_{m, l}(u)[Q(t,u)]^{m+k} \mathrm{~d} u \quad \text{and}\quad h_{m,l}(t)=\frac{k!}{ 2^{m}4^{m+k}(m+k)!} [D^{m}\phi_{m,l}](t).
\end{equation}
Using an argument similar to that in the proof of Theorem \ref{mainthm1}, we obtain
\begin{equation}
     \left[\frac{d}{dt}D^{2m+2k}\phi_{m,l}\right](t)=\frac{(-1)^{m+k}(m+k)!4^{m+k}2^{m+k}}{t^{2(m+k)}}   [\mathcal{L}_{m+k}(t^{2(m+k)+1}\widetilde{f}_{m, l})](1-t).
\end{equation}
    Using \eqref{phiqs}, we obtain
    \begin{equation}
     \frac{ 2^{m}4^{m+k}(m+k)!}{k!}\left[\frac{d}{dt}D^{m+2k}h_{m,l}\right](t)=\frac{(-1)^{m+k}(m+k)!4^{m+k}2^{m+k}}{t^{2(m+k)}}   [\mathcal{L}_{m+k}(t^{2(m+k)+1}\widetilde{f}_{m, l})](1-t),
\end{equation}
that is
 \begin{equation}
  \left[\frac{d}{dt}D^{m+2k}h_{m,l}\right](t)=\frac{(-1)^{m+k}k!2^{k}}{t^{2(m+k)}}   [\mathcal{L}_{m+k}(t^{2(m+k)+1}\widetilde{f}_{m, l})](1-t),
\end{equation}
This completes the proof.
\end{proof}
\section*{\textbf{Appendix}}\label{sec:appendix}
In this section, we establish some combinatorial lemmas that are used in the proof of our main results. In particular, the following lemma plays a key role in the proof of Proposition \ref{prop:mthm1}.
\begin{lemma}\label{ckjr:lemma}
 For $j,k\in\mathbb{N}$ such that $1\leq j\leq k$ and $r\in\mathbb{N}_{0}$ satisfies $0\leq r\leq 2k-2j$, the  expression of $C_{k,j,r}$ in \eqref{ckjr:firstexp} can be simplified to the following form:
    \begin{equation}\label{Ckjr:lemma:exp}
    C_{k,j,r} 
   =(-1)^{r}\frac{k!2^{j}}{j!}\sum_{l=0}^{r+j}\sum_{p=0}^{2k+l-j}2^{p}\binom{r+j}{l}\binom{2k-l}{k-l}\binom{2k+l-j}{p}\binom{2l-j-p}{2l-2j-r-p}.
\end{equation}
\end{lemma}
\begin{proof}
 Consider $C_{k,j,r}$ given in \eqref{ckjr:firstexp}:
 \begin{equation}
      C_{k,j,r}= (-1)^{r}\frac{k!}{j!}\sum_{l=0}^{k-j}2^{k-l}\binom{k+l}{l}\sum_{p=0}^{k-j-l}2^{-p}\binom{2k-j-p-1}{k+l-1}\binom{k-l+p}{p}\binom{p}{r+j-k+l}.
 \end{equation}
Using $\binom{a}{b}\binom{b}{c}=\binom{a}{c}\binom{a-c}{b-c}$, we have
\begin{equation}
    \binom{k-l+p}{p}\binom{p}{p-r-j+k-l}=\binom{k-l+p}{p-r-j+k-l}\binom{r+j}{r+j-k+l}=\binom{k-l+p}{r+j}\binom{r+j}{k-l}.
\end{equation}
This gives
\begin{align}\label{Ckjr:ex}
    C_{k,j,r}
        &=(-1)^{r}\frac{k!}{j!}\sum_{l=0}^{k-j}2^{k-l}\binom{k+l}{l}\binom{r+j}{k-l}\left\{\underbrace{\sum_{p=0}^{k-j-l}2^{-p}\binom{2k-j-p-1}{k-j-l-p}\binom{k-l+p}{r+j}}_{S}\right\}.
\end{align}
Note that the combinatorial factor $ \binom{2k-j-p-1}{k-j-l-p}=0\text{ if either }p>k-j-l ~~ \text{ or } l>k-j~~ \&~~p\geq0.$ Therefore, we can extend the upper limits of $p$ and $l$ to infinity. 
Let us now consider the inner sum $S$. This can be written as a contour integral as follows.
\begin{multline}
     S=\sum_{p=0}^{\infty}2^{-p}\binom{2k-j-p-1}{k-j-l-p}\binom{k-l+p}{r+j}=\\\frac{1}{(2\pi i)^{2}}\int_{|z|=\epsilon_{1}}\frac{(1+z)^{2k-j-1}}{z^{k-j-l+1}}\int_{|w|=\gamma_{1}}\frac{(1+w)^{k-l}}{w^{r+j+1}}\sum_{p=0}^{\infty}\frac{z^{p}(1+w)^{p}}{2^{p}(1+z)^{p}}dwdz.
\end{multline}
If we choose $\epsilon_{1}<\frac{2}{3+\gamma_{1}}$, then 
the series $\sum\limits_{p=0}^{\infty}\frac{z^{p}(1+w)^{p}}{2^{p}(1+z)^{p}}$ converges, and hence we get
\begin{align*}
 S
     &=\frac{1}{(2\pi i)^{2}}\int_{|z|=\epsilon_{1}}\frac{(1+z)^{2k-j-1}}{z^{k-j-l+1}}\int_{|w|=\gamma_{1}}\frac{(1+w)^{k-l}}{w^{r+j+1}}\frac{2(1+z)}{\left(2+z-zw\right)}dwdz.
\end{align*}
By substituting $z$ with $\frac{1}{z}$, we get
\begin{align*}
  S
     &=-\frac{1}{(2\pi i)^{2}}\int_{|z|=\frac{1}{\epsilon_{1}}}\frac{(1+z)^{2k-j-1}z^{k-j-l+1}}{z^{2k-j-1}}\int_{|w|=\gamma_{1}}\frac{(1+w)^{k-l}}{w^{r+j+1}}\frac{2(1+z)z}{z\left(2z+1-w\right)}\frac{1}{z^2}dwdz\\
     &=-\frac{2}{(2\pi i)^{2}}\int_{|z|=\frac{1}{\epsilon_{1}}}\frac{(1+z)^{2k-j}}{z^{k+l}(1+2z)}\int_{|w|=\gamma_{1}}\frac{(1+w)^{k-l}}{w^{r+j+1}}\frac{1}{\left(1-\frac{w}{1+2z}\right)}dwdz,
\end{align*}
where $|z|=\frac{1}{\epsilon_{1}}$ is negatively (clockwise) oriented. Noting that $ \left|\frac{w}{1+2z}\right|<1$, we have
\begin{align*}
   S
     &=-\frac{2}{(2\pi i)^{2}}\int_{|z|=\frac{1}{\epsilon_{1}}}\frac{(1+z)^{2k-j}}{z^{k+l}(1+2z)}\int_{|w|=\gamma_{1}}\frac{1}{w^{r+j+1}}\sum_{s=0}^{k-l}\sum_{p=0}^{\infty}\binom{k-l}{s}w^{s}\frac{w^p}{(1+2z)^p}dwdz.
\end{align*}
Since $w=0$ is the only singularity inside the contour $|w|=\gamma_{1}$, we compute the residue by choosing $p=r+j-s$, and hence we get
\begin{align}
   S
      &=-\frac{2}{2\pi i}\int_{|z|=\frac{1}{\epsilon_{1}}}\frac{(1+z)^{2k-j}}{z^{k+l}(1+2z)}\sum_{s=0}^{k-l}\binom{k-l}{s}\frac{1}{(1+2z)^{r+j-s}}dz\nonumber\\
      &=-\frac{2}{2\pi i}\int_{|z|=\frac{1}{\epsilon_{1}}}\frac{(1+z)^{2k-j}}{z^{k+l}(1+2z)^{r+j+1}}\sum_{s=0}^{k-l}\binom{k-l}{s}(1+2z)^{s}dz\nonumber\\
      &=-\frac{2}{2\pi i}\int_{|z|=\frac{1}{\epsilon_{1}}}\frac{(1+z)^{2k-j}}{z^{k+l}(1+2z)^{r+j+1}}(2+2z)^{k-l}dz\nonumber\\
       &=\frac{-2^{k-l+1}}{2\pi i}\int_{|z|=\frac{1}{\epsilon_{1}}}\frac{(1+z)^{3k-j-l}}{z^{k+l}(1+2z)^{r+j+1}}dz.\label{S:innersum}
\end{align}
Substituting the above expression into \eqref{Ckjr:ex}, we get
\begin{align*}
(-1)^{r}\frac{j!}{k!}C_{k,j,r}
    &=-\sum_{l=0}^{\infty}2^{k-l}\binom{k+l}{l}\binom{r+j}{k-l}\frac{2^{k-l+1}}{2\pi i}\int_{|z|=\frac{1}{\epsilon_{1}}}\frac{(1+z)^{3k-j-l}}{z^{k+l}(1+2z)^{r+j+1}}dz\\
     &=-\sum_{l=0}^{\infty}\binom{k+l}{l}\frac{2^{2k-2l+1}}{(2\pi i)^{2}}\int_{|w|=\widetilde{\gamma}_{1}}\frac{(1+w)^{r+j}}{w^{k-l+1}}\int_{|z|=\frac{1}{\epsilon_{1}}}\frac{(1+z)^{3k-j-l}}{z^{k+l}(1+2z)^{r+j+1}}dzdw\\
&=-\frac{2^{2k+1}}{(2\pi i)^{2}}\int_{|w|=\widetilde{\gamma}_{1}}\frac{(1+w)^{r+j}}{w^{k+1}}\int_{|z|=\frac{1}{\epsilon_{1}}}\frac{(1+z)^{3k-j}}{z^{k}(1+2z)^{r+j+1}}\sum_{l=0}^{\infty}\binom{k+l}{l}\left(\frac{w}{4z(1+z)}\right)^{l}dzdw.
\end{align*}
If we choose $\widetilde{\gamma}_{1}<4\dfrac{(1-\epsilon_{1})}{\epsilon_{1}^{2}}$, then $ \left|\frac{w}{4z(1+z)}\right|<1$ and consequently, we obtain
\begin{align*}
   (-1)^{r}\frac{j!}{k!}C_{k,j,r}
    &=-\frac{2^{2k+1}}{(2\pi i)^{2}}\int_{|w|=\widetilde{\gamma}_{1}}\frac{(1+w)^{r+j}}{w^{k+1}}\int_{|z|=\frac{1}{\epsilon_{1}}}\frac{(1+z)^{3k-j}}{z^{k}(1+2z)^{r+j+1}}\frac{1}{\left(1-\frac{w}{4z(1+z)}\right)^{k+1}}dzdw\\
&=-\frac{2^{4k+3}}{(2\pi i)^{2}}\int_{|w|=\widetilde{\gamma}_{1}}\frac{(1+w)^{r+j}}{w^{k+1}}\int_{|z|=\frac{1}{\epsilon_{1}}}\frac{(1+z)^{4k-j+1}z}{(1+2z)^{r+j+1}}\frac{1}{(4z(1+z)-w)^{k+1}}dzdw.
\end{align*}
Hence, we get
\begin{multline}\label{Int1}
 C_{k,j,r} =(-1)^{r+1}\frac{k!2^{4k+3}}{j!2\pi i}\int_{|z|=\frac{1}{\epsilon_{1}}}\frac{(1+z)^{4k-j+1}z}{(1+2z)^{r+j+1}}\underbrace{\frac{1}{2\pi i}\int_{|w|=\widetilde{\gamma}_{1}}\frac{(1+w)^{r+j}}{w^{k+1}}\frac{1}{(4z(1+z)-w)^{k+1}}dw}_{\mathrm{I}}dz,
\end{multline}
where $|w|=\widetilde{\gamma}_{1}$ and $|z|=\frac{1}{\epsilon_{1}}$ are positively and negatively oriented, respectively.

We first evaluate the inner integral $\mathrm{I}$ using Cauchy residue theorem. Note that   $\widetilde{\gamma}_1=|w|< |4z(1+z)|$, thus $w=0$ lies inside the contour and $w=4z(1+z)$ lies outside the contour. By Cauchy residue theorem, we get
\begin{align}\label{wint:S1}
\mathrm{I}&=\frac{1}{2\pi i(4z(1+z))^{k+1}}    \int_{|w|=\widetilde{\gamma}_{1}}\frac{1}{w^{k+1}}\sum_{l=0}^{r+j}\sum_{p=0}^{\infty}\binom{r+j}{l}\binom{k+p}{p}\frac{w^{l+p}}{(4z(1+z))^{p}}dw\nonumber\\
&=\frac{1}{(4z(1+z))^{k+1}}\sum_{l=0}^{r+j}\binom{r+j}{l}\binom{2k-l}{k-l}\frac{1}{(4z(1+z))^{k-l}}\nonumber\\
&=\sum_{l=0}^{r+j}\binom{r+j}{l}\binom{2k-l}{k-l}\frac{1}{(4z(1+z))^{2k-l+1}}.
\end{align}
Substituting the above expression into \eqref{Int1}, we get
\begin{align}\label{INT3}
 C_{k,j,r} &=(-1)^{r+1}\frac{k!2^{4k+3}}{j!2\pi i}\sum_{l=0}^{r+j}\binom{r+j}{l}\binom{2k-l}{k-l}\int_{|z|=\frac{1}{\epsilon_{1}}}\frac{(1+z)^{4k-j+1}z}{(1+2z)^{r+j+1}(4z(1+z))^{2k-l+1}}dz\nonumber\\
 &=(-1)^{r+1}\frac{k!2^{4k+3}}{j!4^{2k+1}}\sum_{l=0}^{r+j}4^{l}\binom{r+j}{l}\binom{2k-l}{k-l}\underbrace{\frac{1}{2\pi i}\int_{|z|=\frac{1}{\epsilon_{1}}}\frac{(1+z)^{2k+l-j}}{(1+2z)^{r+j+1}z^{2k-l}}dz}_{\mathrm{II}},
\end{align}
where $|z|=\frac{1}{\epsilon_1}$ is a negatively oriented contour, and $\frac{1}{\epsilon_1}>\frac{3}{2}$. Therefore, $z=\infty$ is the only pole inside the domain, and hence
\begin{align*}
\mathrm{II}=\mathrm{Res}(\mathrm{Integrand};z=\infty)
&=\mathrm{Res}\left(-\frac{(1+z)^{2k+l-j}z^{2k-l+r+j+1-2k-l+j-2}}{(2+z)^{r+j+1}};z=0\right)\\
&=\mathrm{Res}\left(-\frac{(1+z)^{2k+l-j}z^{r+2j-2l-1}}{(2+z)^{r+j+1}};z=0\right)\\
&=-\sum_{p=0}^{2k+l-j}\binom{2k+l-j}{p}\binom{2l-j-p}{2l-2j-r-p}\frac{1}{2^{{2l-j-p+1}}}.
\end{align*}
Substituting this into \eqref{INT3}, we get
\begin{align*}
   C_{k,j,r} &=(-1)^{r}\frac{k!2^{4k+3}}{j!4^{2k+1}}\sum_{l=0}^{r+j}4^{l}\binom{r+j}{l}\binom{2k-l}{k-l}\sum_{p=0}^{2k+l-j}\binom{2k+l-j}{p}\binom{2l-j-p}{2l-2j-r-p}\frac{1}{2^{{2l-j-p+1}}}\\
   &=(-1)^{r}\frac{k!2^{j}}{j!}\sum_{l=0}^{r+j}\sum_{p=0}^{2k+l-j}2^{p}\binom{r+j}{l}\binom{2k-l}{k-l}\binom{2k+l-j}{p}\binom{2l-j-p}{2l-2j-r-p}.
\end{align*}
This completes the proof.
\end{proof}
The following lemma is crucially used in the proof of Proposition \ref{prop:thm2}.
\begin{lemma}\label{Bkjr:lemma}
   For $j,k\in\mathbb{N}$ such that $1\leq j\leq k$ and $r\in\mathbb{N}_{0}$ satisfies $0\leq r\leq 2k-2j$, the expression of $B_{k,j,r}$ in \eqref{Bkj:expres} simplifies to $C_{k,j,r}$ as given in \eqref{Ckjr:lemma:exp}.
\end{lemma}
\begin{proof}
    Consider \begin{equation}
    B_{k,j,r}=\sum_{l=0}^{k-j}\frac{(-1)^{r}k!2^j}{(2k-2l-j)(j-1)!}\binom{k+l}{l}\binom{2k-2l-j}{k-l}\binom{2l}{r+j-k+l}.
    \end{equation}
     We rewrite this as
\begin{align}\label{Bkjr:first:exp}
   B_{k,j,r}
    &=(-1)^{r}\frac{k!2^{j}}{j!}\sum_{l=0}^{k-j}\frac{j}{(2k-2l-j)}\binom{k+l}{l}\binom{2k-2l-j}{k-j-l}\binom{2l}{r+j-k+l}.
\end{align}
Note that the binomial $\binom{2k-2l-j}{k-j-l}$ vanishes for $l>k-j$.
Therefore, we can extend the sum to infinity.
Further, using the relation $\frac{n-2k}{n}\binom{n}{k}=\binom{n-1}{k}-\binom{n-1}{k-1}$, we have
\begin{equation}
    \frac{j}{(2k-2l-j)}\binom{2k-2l-j}{k-j-l}=\binom{2k-2l-j-1}{k-j-l}-\binom{2k-2l-j-1}{k-j-l-1}.
\end{equation}
Thus, we have
\begin{align*}
  &   (-1)^{r}\frac{j!}{k!2^{j}}B_{k,j,r}\\&=\sum_{l=0}^{\infty}\binom{k+l}{l}\binom{2l}{r+j-k+l}\left[\binom{2k-2l-j-1}{k-j-l}-\binom{2k-2l-j-1}{k-j-l-1}\right]\\
&=\frac{1}{(2\pi i)^{2}}\sum_{l=0}^{\infty}\binom{k+l}{l}\int_{|w|=\gamma_{2}}\frac{(1+w)^{2l}}{w^{r+j-k+l+1}}\int_{|z|=\epsilon_{2}}\frac{(1+z)^{2k-2l-j-1}(1-z)}{z^{k-j-l+1}}dzdw\\
&=\frac{1}{(2\pi i)^{2}}\int_{|w|=\gamma_{2}}\frac{1}{w^{r+j-k+1}}\int_{|z|=\epsilon_{2}}\frac{(1+z)^{2k-j-1}(1-z)}{z^{k-j+1}}\sum_{l=0}^{\infty}\binom{k+l}{l}\left(\frac{z(1+w)^2}{w(1+z)^2}\right)^{l}dzdw.
\end{align*}
If we choose $\epsilon_{2}<\gamma_{2}<1$ such that $\left|\frac{z(1+w)^2}{w(1+z)^2}\right|<1$, then
 the series $\sum\limits_{l=0}^{\infty}\binom{k+l}{l}\left(\frac{z(1+w)^2}{w(1+z)^2}\right)^{l}$ converges, and the above sum takes the form
\begin{align*}
         (-1)^{r}\frac{j!}{k!2^{j}}B_{k,j,r}&=\frac{1}{(2\pi i)^{2}}\int_{|w|=\gamma_{2}}\frac{1}{w^{r+j-k+1}}\int_{|z|=\epsilon_{2}}\frac{(1+z)^{2k-j-1}(1-z)}{z^{k-j+1}}\frac{1}{\left(1-\frac{z(1+w)^2}{w(1+z)^2}\right)^{k+1}}dzdw.
\end{align*}
Replacing $z$ with $\frac{1}{z}$, we get
\begin{align*}
         (-1)^{r}\frac{j!}{k!2^{j}}B_{k,j,r}&=-\frac{1}{(2\pi i)^{2}}\int_{|w|=\gamma_{2}}\frac{1}{w^{r+j-k+1}}\int_{|z|=\frac{1}{\epsilon_{2}}}\frac{(1+z)^{2k-j-1}(z-1)}{z^{2k-j-1-k+j-1+1}}\frac{1}{\left(1-\frac{z(1+w)^2}{w(1+z)^2}\right)^{k+1}}\frac{1}{z^2}dzdw\\
&=-\frac{1}{(2\pi i)^{2}}\int_{|w|=\gamma_{2}}\frac{1}{w^{r+j-2k}}\int_{|z|=\frac{1}{\epsilon_{2}}}\frac{(1+z)^{4k-j+1}(z-1)}{z^{k+1}}\frac{1}{[(w-z)(1-wz)]^{k+1}}dzdw,
\end{align*}
where $|w|=\gamma_{2}$ and $|z|=\frac{1}{\epsilon_{2}}$ are positively and negatively oriented contours, respectively.
If we replace $z$ with $2z+1$, then  we get
\begin{multline}
     B_{k,j,r}\\
=(-1)^{r+1}\frac{k!2^{4k+3}}{j!(2\pi i)^{2}}\int_{|2z+1|=\frac{1}{\epsilon_{2}}}\frac{(1+z)^{4k-j+1}z}{(2z+1)^{k+1}}\int_{|w|=\gamma_{2}}\frac{1}{w^{r+j-2k}}\frac{1}{[(w-2z-1)(1-w(2z+1))]^{k+1}}dzdw.
\end{multline}
Making the change of variable $w=\frac{1+2z}{1+\tilde{w}}$, and rewriting $\tilde{w}$ as $w$, we get
\begin{multline}\label{Bkjr:int}
     B_{k,j,r}\\=(-1)^{r}\frac{k!2^{4k+3}}{j!2\pi i}\int_{|2z+1|=\frac{1}{\epsilon_{2}}}\frac{(1+z)^{4k-j+1}z}{(2z+1)^{r+j+1}}\underbrace{\frac{1}{2\pi i}\int_{\left|\frac{1+w}{1+2z}\right|=\frac{1}{\gamma_{2}}}\frac{(1+w)^{r+j}}{w^{k+1}}\frac{1}{[\left(4z(z+1)-w\right)]^{k+1}}dw}_{\mathrm{I}}dz,
\end{multline}
where both the contours $\left|\frac{1+w}{1+2z}\right|=\frac{1}{\gamma_{2}}$ and $|2z+1|=\frac{1}{\epsilon_{2}}$ are negatively oriented.
We now evaluate the inner integral $\mathrm{I}$ using the Cauchy residue theorem.
Note that the contour for $w$ is $|1+w|=\frac{|2z+1|}{\gamma_{2}}=\frac{1}{\epsilon_{2}{\gamma_{2}}}>1$, since $\epsilon_{2}<\gamma_{2}<1$. Therefore, $w=0$ lies outside the contour and $w=4z(1+z)$ lies inside the contour.
Now using the Cauchy residue theorem, similar to \eqref{wint:S1}, we get
\begin{align*}
    \mathrm{I}&=-\mathrm{Res}(\mathrm{Integrand};w=0)=-\sum_{l=0}^{r+j}\binom{r+j}{l}\binom{2k-l}{k-l}\frac{1}{(4z(1+z))^{2k-l+1}}.
\end{align*}
Substituting this expression into \eqref{Bkjr:int}, we have
\begin{align*}
   B_{k,j,r}&=(-1)^{r+1}\frac{k!2^{4k+3}}{j!2\pi i}\sum_{l=0}^{r+j}\binom{r+j}{l}\binom{2k-l}{k-l}\int_{|2z+1|=\frac{1}{\epsilon_{2}}}\frac{(1+z)^{4k-j+1}z}{(2z+1)^{r+j+1}}\frac{1}{(4z(1+z))^{2k-l+1}}\\
   &=(-1)^{r+1}\frac{k!2^{4k+3}}{j!2\pi i4^{2k+1}}\sum_{l=0}^{r+j}4^{l}\binom{r+j}{l}\binom{2k-l}{k-l}\int_{|2z+1|=\frac{1}{\epsilon_{2}}}\frac{(1+z)^{2k+l-j}}{(1+2z)^{r+j+1}z^{2k-l}}dz,
\end{align*}
where $|2z+1|=\frac{1}{\epsilon_2}$ is a negatively oriented contour and $\frac{1}{\epsilon_2}>1$. Therefore, $z=\infty$ is the only pole inside the domain, and hence the inner integral gives
\begin{align*}
    \frac{1}{2\pi i}\int_{|2z+1|=\frac{1}{\epsilon_{2}}}\frac{(1+z)^{2k+l-j}}{(1+2z)^{r+j+1}z^{2k-l}}dz&=\mathrm{Res}(\mathrm{Integrand};z=\infty)\\
    &=-\sum_{p=0}^{2k+l-j}\binom{2k+l-j}{p}\binom{2l-j-p}{2l-2j-r-p}\frac{1}{2^{{2l-j-p+1}}}.
\end{align*}
Substituting this expression back into $B_{k,j,r}$, we finally get
\begin{align*}
      B_{k,j,r} 
   =(-1)^{r}\frac{k!2^{j}}{j!}\sum_{l=0}^{r+j}\sum_{p=0}^{2k+l-j}2^{p}\binom{r+j}{l}\binom{2k-l}{k-l}\binom{2k+l-j}{p}\binom{2l-j-p}{2l-2j-r-p}=C_{k,j,r}.
\end{align*}
This completes the proof of the lemma.
\end{proof}
\section*{Acknowledgements}
\noindent PC and NS acknowledge the support of the Department of Atomic Energy, Government of India, for the Ph.D. fellowship.\\
AT acknowledges the support of the Anusandhan National Research Foundation (ANRF), under the National Post-Doctoral Fellowship scheme, file no. PDF/2025/005843.\\
Finally, the authors would like to sincerely  thank Prof. Venkateswaran P. Krishnan for several helpful discussions and suggestions.
\bibliographystyle{plain}
\bibliography{references}
  \end{document}